\documentclass[12pt,oneside,english]{amsart}
\usepackage[latin9]{inputenc}
\usepackage[a4paper]{geometry}
\usepackage{graphicx,color,stmaryrd}
\usepackage[dvipsnames]{xcolor}
\usepackage{amstext}
\usepackage{amsthm}
 \usepackage{amsmath}
\usepackage{appendix}
\usepackage{hyperref}
\usepackage{comment}

\makeatletter
\numberwithin{equation}{section}
\numberwithin{figure}{section}
\theoremstyle{plain}
\newtheorem{thm}{\protect\theoremname}
  \theoremstyle{definition}
  \newtheorem{defn}[thm]{\protect\definitionname}
  \theoremstyle{definition}
  \newtheorem{example}[thm]{\protect\examplename}
  \theoremstyle{plain}
  \newtheorem{prop}[thm]{\protect\propositionname}
  \theoremstyle{plain}
  \newtheorem{lem}[thm]{\protect\lemmaname}

\newcommand{\R}{\mathbb{R}}

\newcommand{\C}{\mathbb{C}}

\newcommand{\cH}{\mathcal{H}}
\newcommand{\cD}{\mathcal{D}}

\makeatother

\makeatother

\usepackage{babel}
  \providecommand{\definitionname}{Definition}
  \providecommand{\examplename}{Example}
  \providecommand{\lemmaname}{Lemma}
  \providecommand{\propositionname}{Proposition}
\providecommand{\theoremname}{Theorem}

\begin{document}
\title{Geometric States\\{\it In memory of Steve Zelditch}}

\author{Benoit Dherin} 
\address{Google, Mountain View, CA 94043, USA}
\email{dherin@google.com}
\author{Alan Weinstein}
\address{Department of Mathematics, University of California, Berkeley, CA 94720, USA, and Department of Mathematics, Stanford University, Stanford,  CA 94305, USA}
\email{alanw@math.berkeley.edu}

\subjclass[2020]{46C50, 46F10, 53C99, 81S10}

\maketitle
\tableofcontents{}

\date{January 2024}

\maketitle

\begin{abstract}
We introduce a special family of distributional $\alpha$-densities
and give a trans\-versality criterion stating when their product is
defined, closely related to H\"ormander's criterion for general distributions.
Moreover, we show that for the subspace of distributional half-densities
in this family the distribution product naturally yields a pairing
that extends the usual one on smooth half-densities.
\end{abstract}

\section{Introduction}

The notion of state in quantum mechanics is generally defined as a unit vector, or ray, in a Hilbert space \cite{isham2001}, and the probability of transition from one state $\psi_1$  to another state $\psi_2$  is given using the  Hilbert space inner product as $|\langle \psi_1,\,  \psi_2 \rangle|^2$. When these states are realized as the solutions of the Schr\"odinger equation a number of technical problems arise with the Hilbert space formulation of the states. In particular,  since its solutions (or initial conditions) can  have singularities, they are distributions rather than honest functions in a $L^2$  Hilbert space \cite{isham2001}. The necessity to include certain distributions as states makes the axiomatic requirement of the state space to be a Hilbert space difficult to fulfill in practice. This has led to considering weaker definitions than that of a Hilbert space as natural  state spaces for quantum mechanics. For example,  partial inner product spaces \cite{antoine1976, antoine2009} are vector spaces where the inner product (and hence the transition probability) is only partially defined. This issue is closely related to the product of distributions, which is well-known not to be possible in general \cite{brouder2016,hormander2009}.   H\"ormander gave a criterion involving  properties of the distributions' wave-front sets governing when their product is possible \cite{hormander2009}. Extending the product from functions to distributions has also important implications in quantum field theory and renormalization theory; see \cite{dang2013,dang2016} for example.

In this note, we take a geometric approach to the notion of state rather than an analytic one, based on the concept of $\alpha$-densities as described in \cite{BW1997,GS1977}. For instance, the space of half densities on a manifold $X$ is a natural candidate for a quantum state space as it comes equipped with an intrinsic Hilbert space product (upon taking suitable completions) and locally it can be modeled as the space of square-integrable functions. To include singular states, we still need to consider the dual space of distributional half densities though on which the natural intrinsic inner product is unfortunately not always defined. 

The main contribution of this work is to give an explicit geometric description of a special family of distributional $\alpha$-densities, which we call {\it geometric states}, for which the product is well-defined provided that their spatial components (which we call {\it cores}) intersect transversally (which is a special case of H\"ormander's criterion \cite{hormander2009}). When specialized to $\alpha = 1/2$, this product allows us to extend the intrinsic inner-product of half densities to the whole space of geometric states, assuming transversal intersection of their cores.

\section{Geometric States}

Before giving a definition of a geometric state, let us start by recalling the notion of $\alpha$ -densities; see \cite{BW1997,GS1977} for more details.
Consider a $n$-dimensional vector space $V$ over $\mathbb C$. We denote by $V^*$ its dual, that is, the vector space of linear maps $V\rightarrow \mathbb C$, and by $F(V)$ its set of frames, that is, the set of ordered bases ${\bf e}=(e_1, \dots, e_n)$. For each $\alpha\in \mathbb C$, the space of $\alpha$-densities on $V$ is the space of mappings $\eta: F(V) \rightarrow \mathbb C$  such that  $\eta( A{\bf e}) = |\det A|^\alpha\eta ({\bf e})$, 
where $A$ is a non-singular linear map and $|\det A|$ is
the absolute value of its determinant. We denote by $|V|^\alpha$ the space of $\alpha$-densities on $V$, which is a complex vector space of dimension 1.

Now consider a smooth manifold $X$. The space of $\alpha$-densities on $X$, which we denote by $|\Omega|^\alpha(X)$, is the space of sections of the line-bundle $|TX|^\alpha \rightarrow X$ whose fiber at $x$ is $|T_x X|^\alpha$.  More generally, given a vector bundle $E\rightarrow X$, we can form the line bundle $|E|^\alpha\rightarrow X$ whose fiber at $x$ is the 1-dimensional vector space  over $\mathbb C$ of $\alpha$-densities $|E_x|^\alpha$. We consider the space of smooth sections $\Gamma(X,|E|^\alpha)$  as a $C^\infty(X)$-module. We will consider the tensor product of such modules over $C^\infty(X)$. There are a few canonical isomorphisms of importance to us that we now enumerate for convenience. First of all $|E^*|^\alpha\rightarrow X$ naturally identifies with $|E|^{-\alpha}\rightarrow X$.  Second, given an exact sequence  $0\rightarrow A\rightarrow B\rightarrow C\rightarrow 0$ of vector bundles over $X$, we have the canonical isomorphism $|B|^\alpha \simeq |A|^\alpha \otimes |C|^\alpha$,  where we use the shorthand notation $|E|^\alpha$ to denote the vector bundle $|E|^\alpha \rightarrow X$. At last, we have that $|E|^\alpha \otimes |E|^\beta\simeq |E|^{\alpha + \beta}$ which specializes to $|\Omega |^\alpha(X) \otimes |\Omega|^\beta(X) \simeq |\Omega|^{\alpha + \beta}(X)$ when $E$ is the tangent bundle to $X$. 

We can now define our space of geometric states:
\begin{defn}
We define the space $\cH_{X,\,C}^{\alpha}$, where $C$ is a smooth
submanifold of a smooth manifold $X$ and $\alpha\in\C$, as the
subspace of the following density bundle sections
\[
|\Omega|^{\alpha}(C)\otimes\Gamma(|N^{*}C|^{1-\alpha}).
\]
The first factor
of the tensor product above is the space of smooth $\alpha$-densities
on $C$, while the second factor is the space of the smooth sections
of the $(1-\alpha)$-density bundle associated to the conormal bundle
$N^{*}C\rightarrow C$ to $C$. We call an element in either of this
spaces interchangeably a \textbf{geometric state} or a \textbf{geometric
distributional $\alpha$-density}.
\end{defn}

This family of spaces above has interesting well-known extreme cases:

\begin{example}
When $C=X,$ the spaces $\cH_{X,\,C}^{\alpha}$ coincide with the usual
smooth $\alpha$-densities on the full space $X$. In particular,
when $\alpha=0$, this space identifies with the space $\mathcal{E}_{X}$
of the smooth functions on $X$, which, in turn, can be regarded as
a subspace of the distributions on $X$ (with test ``functions" taken
in the smooth compactly supported 1-densities on $X$). 
\end{example}

\begin{example}
When $\alpha=1,$ we have the identification of $\cH_{X,\,C}^{\alpha}$
with the space of smooth $1$-densities (or measures) supported on
the submanifold $C$. This space can also be regarded as a subspace
of the distributions on $X$ supported on $C$ (now with test functions
the compactly supported smooth functions on $X$). 
\end{example}

\begin{example}
When $C$ is reduced to a single point $x\in X$, then $\cH_{X,\,C}^{\alpha}$
identifies with the $(\alpha-1)$-densities $|T_{x}X|^{\alpha-1}$ at
that point, which further identifies with $\R$ when $\alpha=1$, the
inverse volume elements $1/|T_{x}X|$ when $\alpha=0$ (which will see are
related with the delta distributions), and the volume elements  $|T_{x}X|$
when $\alpha=2$. 
\end{example}

A last example comes from the following Proposition

\begin{prop}\label{prop:conormal_state}
Let $C$ be a smooth submanifold of  $X$. Consider
the conormal bundle $N^{*}C$ to $C$ in $X$. Then the restriction
of a half-density in $|\Omega|^{\frac{1}{2}}(N^{*}C)$ to the zero
section of $N^{*}C\rightarrow C$ can be identified with an element
of $|\Omega|^{\frac{1}{2}}(C)\otimes\Gamma(|N^{*}C|^{\frac{1}{2}})$, that is, with an element of $\cH_{X,\,C}^{\frac{1}{2}}$.
\end{prop}

\begin{proof}
This comes directly from the fact that the tangent space to $N^{*}C$
restricted to the zero section can be identified with the bundle $TC\oplus N^{*}C$
and the usual canonical isomorphism for $\alpha$-density bundles.
\end{proof}
The proposition above motivates the following example:

\begin{example}
Half-densities on lagrangian submanifolds, such as the conormal bundle $N^{*}C$ to a submanifold $C$, are typically related to some form of ``geometric quantum states'' (see \cite{BW1997} for example). As shown in Proposition \ref{prop:conormal_state}, a half-density on $N^*C$ induces a element of $\cH_{X,\,C}^{\frac{1}{2}}$. This space can be regarded as a subspace of the distributional half-densities
on $X$ thanks to Proposition \ref{prop:pairing}, which makes it clear that
the pairing of an element of $\cH_{X,C}^{\frac{1}{2}}$ with a compactly
supported half density in $|\Omega_{0}|^{\frac{1}{2}}(X)$ yields
a compactly supported $1$-density on $C$ that can be integrated
over $C$ into a complex number. This yields the natural pairing:
\[
\langle\cdot\,,\cdot\rangle\::\;\cH_{X,\,C}^{\frac{1}{2}}\times|\Omega_{0}|^{\frac{1}{2}}(X)\longrightarrow\C.
\]
The examples above suggest the name of ``geometric states'' for the
spaces $\cH_{X,C}^{\alpha}$, since the $1$-density example is related
to classical state spaces, while the half-density example is related
to the quantum state spaces, both of which are given here in term
of geometrical data. 
\end{example}

\section{Geometric states as distributions}
In this section,
we show
that elements in the geometric state space $\cH_{X,\,C}^{1-\alpha}$
naturally identify with distributional 
$(1-\alpha)$-densities, that is,
elements in the dual to the space of the compactly supported smooth
$\alpha$-densities on $X$.
In the following, we will denote the
spaces of compactly supported smooth $\alpha$-densities on $X$ by
\[
\cD_{X}^{\alpha}:=|\Omega_{0}|^{\alpha}(X).
\]
Note that $\cD_{X}^{0}$ identifies with the space $\cD_{X}$ of compactly
supported functions (test functions) on $X$. 
\begin{lem}
\label{prop:restriction}Let $C$ be a smooth submanifold of a smooth
manifold $X$. The restriction to $C$ of the $\alpha$-densities
on $X$ yields the following canonical map
\begin{equation}
|\Omega|^{\alpha}(X)\longrightarrow|\Omega|^{\alpha}(C)\otimes\Gamma(|NC|^{\alpha})
\end{equation}
\end{lem}
\begin{proof}
The proof follows immediately from the exact sequence $0\rightarrow TC\rightarrow TX\rightarrow NC\rightarrow0$
and the canonical isomorphism $|A\oplus B|^{\alpha}\simeq|A|^{\alpha}\otimes|B|^{\alpha}$.\end{proof}
\begin{prop}\label{prop:pairing}
There is a natural bilinear pairing
\[
\langle\cdot,\cdot\rangle:\;\cH_{X,C}^{1-\alpha}\times\cD_{X}^{\alpha}\longrightarrow\C,
\]
turning $\cH_{X,C}^{1-\alpha}$ into a subset of the distributional
$(1-\alpha)$-densities $(\cD_{X}^{\alpha})^{'}$ on $X$. \end{prop}
\begin{proof}
Thanks to Lemma \ref{prop:restriction}, the restriction of an element
$\phi\in\cD_{X}^{\alpha}$ to the submanifold $C$ can be identified
with a compactly supported section 
\[
\phi_{|C}\in|\Omega_{0}|^{\alpha}(C)\otimes\Gamma_{0}(|NC|^{\alpha}).
\]
Now taking the tensor product of this restriction with a section
\[
\theta\in|\Omega|^{1-\alpha}(C)\otimes\Gamma(|N^{*}C|^{\alpha})
\]
 (i.e. with an element $\theta\in\cH_{X,C}^{1-\alpha}$), we obtain
a compactly supported $1$-density
\[
\theta\otimes\phi_{|C}\in|\Omega_{0}|^{1}(C)
\]
Hence, we can integrate this tensor product, yielding the following
natural pairing:
\[
\langle\theta,\,\phi\rangle:=\int_{C}\theta\otimes\phi_{|C}.
\]

\end{proof}

\section{Product of geometric distributions}

On the $\alpha$-density test spaces $\cD_{X}^{\alpha}$, that is,
the spaces of compactly supported smooth $\alpha$-densities on $X$,
we have the natural product of $\alpha$-densities
\[
\cD_{X}^{\alpha}\times\cD_{X}^{\beta}\longrightarrow\cD_{X}^{\alpha+\beta}
\]
given by the tensor product of the density bundle sections. (When $\alpha=\beta=0$,
it degenerates to the product of smooth compactly supported functions.)
The question of whether and when this product can be extended to the
distributional spaces $(\cD_{X})^{'}$ has wide-reaching implications
for both partial differential equations and quantum-field theory.
In general, it is not possible, and H\"ormander \cite{hormander2009} 
has given a criterion on
pairs of distributions which is sufficient to multiply them
together in a meaningful way. Roughly, the H\"ormander criterion states
that the product of distributions makes sense when their wave front sets
do not interact. 

Here we are considering the spaces of distributional $\alpha$-densities
(not only the case $\alpha=0$), and we specialize to the case of
the distributional $\alpha$-density spaces introduced in the previous
sections. The main result is that the product of two geometric states
in $\cH_{X,C}^{\alpha}$ and in $\cH_{X,D}^{\beta}$ is defined when
$C$ and $D$ intersect transversally along a smooth submanifold.
Let's start to see that it is true on a particular sub-family of the
geometric states

Consider the degenerate case when $C=D=X$. Here, we obviously have
a transverse intersection. Now, we have that $\cH_{X,X}^{\alpha}$
coincides with the smooth $\alpha$-densities on $X$ and $\cH_{X,D}^{\beta}$
coincides with the smooth $\beta$-densities on $X$. In this case,
we already know that the tensor product gives a well defined product
\[
\cH_{X,X}^{\alpha}\times\cH_{X,X}^{\beta}\rightarrow\cH_{X,X}^{\alpha+\beta}
\]
thanks to the canonical isomorphism
$|\Omega|^{\alpha}(X)\otimes|\Omega|^{\beta}(X)\simeq|\Omega|^{\alpha+\beta}(X).$
The following theorem states that this product can be also extended
to geometric states, provided that their supports intersect
transversally. Moreover, in the particular case when $\alpha=\beta=1/2$,
this product can be used to define a pairing, since the
product between two geometric distributional half-densities specializes to a $1$-density on the intersection, which can then be integrated, resulting into a pairing between geometric half-densities with transverse intersection.

\begin{thm}
Let $C$ and $D$ be two smooth submanifolds of a manifold $X$ with
smooth transverse intersection $C\cap D$. Then, we have the canonical
mapping 
\[
\cH_{X,\,C}^{\alpha}\otimes\cH_{X,\,D}^{\beta}\, \rightarrow\,\cH_{X,\,C\cap D}^{\alpha+\beta}
\]
defining a product on the space of geometric distributional densities.
\end{thm}

\begin{proof}
The proof rests on two simple properties. The first one is always
valid, and it was already implicitly used in the proof of Lemma \ref{prop:restriction}.
We state it explicitly now at the level of the density bundles:
\begin{equation}
|TX_{|U}|^{\delta}\simeq|TU|^{\delta}\otimes|NU|^{\delta},\label{eq:decomposition}
\end{equation}
for any smooth submanifold $U$ of smooth manifold $X$. The second
property is valid only when the intersection $C\cap D$ of the two
smooth submanifolds $C$ and $D$ of $X$ is smooth and transverse;
in this case, we have that
\begin{equation}
|N(C\cap D)|^{\delta}\simeq|NC|^{\delta}\otimes|NC|^{\delta}.\label{eq:transversality}
\end{equation}
Now consider $\theta_{1}\in\cH_{X,C}^{\alpha}$ and $\theta_{2}\in\cH_{X,D}^{\beta}$.
The tensor product $\theta_{1}\otimes\theta_{2}$ of the restriction
to the intersection $C\cap D$ is a section of the following tensor
products of density bundles
\begin{equation}
|TC|^{\alpha}\otimes|TD|^{\beta}\otimes|N^{*}C|^{1-\alpha}\otimes|N^{*}D|^{1-\beta}.\label{eq:productspace}
\end{equation}
Using \eqref{eq:decomposition}, we can decompose the first two factors
as
\[
|TC|^{\alpha}\simeq|TX|^{\alpha}\otimes|N^{*}C|^{\alpha}\quad\textrm{and}\quad|TD|^{\beta}\simeq|TX|^{\beta}\otimes|N^{*}D|^{\beta}.
\]
Now substituting the new expression of these factors into \eqref{eq:productspace},
grouping the terms together, and using the usual canonical isomorphisms
of density bundles, we obtain that the product space \eqref{eq:productspace}
identifies with
\[
|TX|^{\alpha+\beta}\otimes|N^{*}C|^{1}\otimes|N^{*}D|^{1}.
\]
Since we are assuming that $C$ and $D$ are transverse, we can then
use the identification \eqref{eq:transversality} for the two last
factors above yielding
\[
|TX|^{\alpha+\beta}\otimes|N^{*}(C\cap D)|^{1}.
\]
Now, using again the decomposition \eqref{eq:decomposition}, we obtain
that \eqref{eq:productspace} can further be identified with
\[
|T(C\cap D)|^{\alpha+\beta}\otimes|N^{*}(C\cap D)|^{1-(\alpha+\beta)}.
\]
Thus, the tensor product $\theta_{1}\otimes\theta_{2}$ of the restrictions
coincide with a section of this last bundle, that is, with an element
of $\cH_{X,C\cap D}^{\alpha+\beta}$. 
\end{proof}

\section{Pairing of geometric distributions}

The space of all compactly supported smooth $\alpha$-densities does not form an inner-product space (unless $\alpha = \frac{1}{2}$) but there is always a pairing 
$$\langle \cdot\, , \cdot \rangle\,: \cD_X^\alpha \, \times \, \cD_X^{1-\alpha}\longrightarrow \C,$$
where the pairing between a (compactly supported) $\alpha$-density and a (compactly supported) $(1-\alpha)$-density comes from integrating the 1-density obtained by taking their  product. The global structure on $D_X = \oplus_{\alpha}D_X^\alpha$, on which the inner product is only partially defined, is reminiscent of that of a \textbf{Partial Inner-Product Space} introduced in \cite{antoine1976} (see also \cite{antoine2009}).

The main theorem in the previous section tells us that this structure can be extended to the space of geometric distributions. Namely, if we restrict the geometric state product to pairs of geometric subspaces with density degree in the same relation as above and with cores intersecting transversally, we obtain the following specialization:

\[
\cH_{X ,\, C}^\alpha \,\otimes \,\cH_{X,\, D}^{1-\alpha} \,\rightarrow\, 
\cH_{X, \, C\cap D}^1 \,\simeq \, |\Omega |^1(C\cap D).
\]
Now, integrating this last $1$-density, we obtain a pairing that generalizes the Partial Inner-Product space structure to the space of geometric distributional $\alpha$-densities. Let us summarize this in a proposition:

\begin{prop}
Consider the geometric distribution spaces $\cH^\alpha_{X,\, C}$ and $\cH^\beta_{X, \, D}$. If we have further that $\alpha + \beta = 1$ and that $C$ and $D$ intersect transversally along a compact submanifold, there is a natural bilinear pairing between these spaces: 
\[
\langle \cdot\, , \cdot \rangle\,:
\cH_{X ,\, C}^\alpha \,\times \,\cH_{X,\, D}^{1-\alpha}
\longrightarrow \C.
\]
In particular, these pairings endow the geometric distributional half-densities ($\alpha=1/2$) with a partially defined inner-product, extending the natural inner-product on the space of smooth compactly supported half-densities. 
\end{prop}

If we regard the smooth half-density space $\cD_X^{1/2}$ on $X$ as modeling the states of a quantum system, then the collection of geometric distribution spaces
$$\cH_X^{1/2}:=\oplus_{C}\cH_{X,C}^{\frac 12}$$ extends this model. The partially-defined inner product should be then thought as giving the probability transition from one geometric state to another, under the compatibility condition that the cores of these states are intersecting transversally. This implies that a geometric state spatially supported on $C$ can transition to a geometric state spatially supported on $D$ only  when these states have some amount of  spatial overlap.

\end{document}